\documentclass[draft,1p,times,authoryear]{elsarticle}
\usepackage{booktabs} 
\usepackage{amsmath}
\usepackage{amsthm}
\usepackage{amssymb}
\usepackage{amsfonts}
\usepackage{float}
\usepackage[normalem]{ulem}
\usepackage{algorithm}
\usepackage[noend]{algpseudocode}
\usepackage[utf8]{inputenc}
\usepackage{enumitem}
\usepackage{graphicx}
\usepackage{colortbl}
\usepackage{tikz}
\usetikzlibrary{arrows,matrix,fit,positioning}

\usepackage{listings}
\newtheorem{theorem}{Theorem}
\newtheorem{lemma}[theorem]{Lemma}
\newtheorem{corollary}[theorem]{Corollary}

\newtheorem{proposition}[theorem]{Proposition}
\newtheorem{definition}[theorem]{Definition}

\newtheorem{example}[theorem]{Example}
\newtheorem{remark}[theorem]{Remark}

\setlist{nosep} 

\usepackage{xspace}
\usepackage{numprint}

\usepackage[colorinlistoftodos]{todonotes}
\usepackage{marginnote}

\newcommand{\nfn}{\textsf{NF}\xspace}
\newcommand{\magma}{\textsc{Magma}\xspace}

\newcommand{\julia}{\textsc{Julia}\xspace}
\newcommand{\singular}{\textsc{Singular}\xspace}
\newcommand{\oscar}{\textsc{OSCAR}\xspace}
\newcommand{\gbjl}{\textsc{GB.jl}\xspace}
\newcommand{\gbl}{\textsc{GB}\xspace}

\usepackage{url,amsmath}
\usepackage{booktabs}
\usepackage{multirow}

\newcommand*{\defeq}{\mathrel{\vcenter{\baselineskip0.5ex \lineskiplimit0pt
                     \hbox{\scriptsize.}\hbox{\scriptsize.}}}%
                     =}

\DeclareMathOperator{\kidx}{\mathcal{K}}
\DeclareMathOperator{\lidx}{\mathcal{L}}

\DeclareMathOperator{\headSym}{lt}
\DeclareMathOperator{\headmSym}{lm}

\DeclareMathOperator{\lcm}{lcm}
\newcommand{\n}{\ensuremath{n}\xspace}
\newcommand{\lc}[1]{\mathrm{lc}\left({#1}\right)}
\newcommand{\hd}[1]{\headSym\left({#1}\right)}
\newcommand{\hm}[1]{\headmSym\left({#1}\right)}

\newcommand{\tail}[1]{\ensuremath{\mathrm{tail}\left({#1}\right)}}

\newcommand{\ann}[1]{\ensuremath{\mathrm{Ann}\left({#1}\right)}}

\newcommand{\spoly}[2]{\ensuremath{\mathrm{spoly}\left({#1},{#2}\right)}}
\newcommand{\gpoly}[2]{\ensuremath{\mathrm{gpoly}\left({#1},{#2}\right)}}
\newcommand{\apoly}[1]{\ensuremath{\mathrm{apoly}\left({#1}\right)}}
\newcommand{\nf}[2]{\ensuremath{\text{\sffamily{NF}}\left({#1},{#2}\right)}}

\newcommand{\R}{\ensuremath{\mathcal{R}}\xspace}
\newcommand{\PR}{\ensuremath{\mathcal{R}[x]}\xspace}

\newcommand{\const}{a}

\newcommand{\spt}{S-poly\-no\-mial\xspace}
\newcommand{\spts}{S-poly\-no\-mials\xspace}
\newcommand{\gpt}{GCD-poly\-no\-mial\xspace}
\newcommand{\gpts}{GCD-poly\-no\-mials\xspace}

\newcommand{\apts}{anni\-hi\-la\-tor poly\-no\-mials\xspace}

\newcommand{\stbs}{standard bases\xspace}
\newcommand{\sgb}{strong Gr\"ob\-ner ba\-sis\xspace}
\newcommand{\sgbs}{strong Gr\"ob\-ner ba\-ses\xspace}
\newcommand{\gb}{Gr\"ob\-ner ba\-sis\xspace}
\newcommand{\gbs}{Gr\"ob\-ner ba\-ses\xspace}

\newcommand{\sbba}{{\sffamily sBBA}\xspace}

\newcommand{\F}{\ensuremath{\mathbb{F}}}
\newcommand{\Q}{\ensuremath{\mathbb{Q}}}
\newcommand{\Z}{\ensuremath{\mathbb{Z}}}
\newcommand{\Zn}{\ensuremath{\Z_{\n}}}
\newcommand{\N}{\ensuremath{\mathbb{N}}}

\definecolor{mygreend}{HTML}{2ca92c}
\definecolor{myredd}{HTML}{c31313}

\usepackage{caption}
\definecolor{codegreen}{rgb}{0,0.6,0}
\definecolor{codegray}{rgb}{0.5,0.5,0.5}
\definecolor{codepurple}{rgb}{0.58,0,0.82}
\definecolor{backcolour}{rgb}{0.99,0.99,0.99}

\renewcommand*{\arraystretch}{1.5}

\setlist[enumerate]{topsep=1ex, itemsep=1ex}

\makeatletter
\newenvironment{breakablealgorithm}
  {
   \begin{center}
     \refstepcounter{algorithm}
     \hrule height.8pt depth0pt \kern2pt
     \renewcommand{\caption}[2][\relax]{
       {\raggedright\textbf{\ALG@name~\thealgorithm} ##2\par}%
       \ifx\relax##1\relax 
         \addcontentsline{loa}{algorithm}{\protect\numberline{\thealgorithm}##2}%
       \else 
         \addcontentsline{loa}{algorithm}{\protect\numberline{\thealgorithm}##1}%
       \fi
       \kern2pt\hrule\kern2pt
     }
  }{
     \kern2pt\hrule\relax
   \end{center}
  }
\makeatother

\begin{document}
\title{Efficient Gr\"obner Bases Computation over Principal Ideal Rings }
%

\author{Christian Eder}
\address{University of Leipzig\\Department of Mathematics\\D-04109 Leipzig}
\ead{eder@math.uni-leipzig.de}

\author{Tommy Hofmann}
\address{Technische Universtität Kaiserslautern\\Department of Mathematics\\D-67663 Kaiserslautern}
\ead{thofmann@mathematik.uni-kl.de}

\begin{abstract}
In this paper we present a new efficient variant to compute \sgbs over quotients of principal
ideal domains. We show an easy lifting process which allows us to reduce one
computation over the quotient $\R/\n\R$ to two computations over $\R/a\R$ and
$\R/b\R$ where $\n = ab$ with coprime $a, b$. Possibly using available factorization algorithms we
may thus recursively reduce some \sgb computations to \gb computations over fields
for prime factors of $\n$, at least for squarefree $n$. Considering now a
computation over $\R/n\R$ we can run
a standard \gb algorithm pretending $\R/n\R$ to be field. If we discover a non-invertible leading
coefficient $c$, we use this information to try to split $n = ab$ with coprime $a, b$. If no such $c$ is discovered, the
returned \gb is already a \sgb for the input ideal over $\R/n\R$.
\end{abstract}

%
%
%


\maketitle

\section{Introduction}
\label{sec-intro}
In $1964$ Hironaka already investigated
computational approaches towards singularities and introduced the notion of
\stbs for local monomial orders, see, for
example,~\cite{hironaka11964, hironaka21964, grauert1972}.
In~\cite{bGroebner1965, bGroebner1965eng}, Buchberger initiated, in $1965$, the theory
of \gbs for global monomial orders by which many fundamental problems in mathematics, science
and engineering can be solved algorithmically. Specifically, he introduced some key
structural theory, and based on
this theory, proposed the first algorithm for computing \gbs.
Buchberger's algorithm introduced the concept of critical pairs and repeatedly carries out a certain
polynomial operation (called reduction).


Once the underlying structure is no longer a field, one needs
the notion of strong \gbs respectively strong \stbs. Influential work was done
by~\cite{kapur1988}, introducing the first generalization of Buchberger's
algorithm over Euclidean domains computing strong \gbs. Since then only a few
optimizations have been introduced, see, for example,~\cite{Wienand2011,
lichtblau2012, eppSigZ2017}. For more general rings, like principal ideal
domains or rings, more recent approaches can be found, for example,
in~\cite{norton_2001, pauer-2007, Popescu2016, francis-verron-2019}. Common to all these
approaches is the idea to transfer ideas from the well studied field case, like
criteria for predicting zero reductions or the use of linear algebra, to the
setting of rings.

In some sense, we take this approach to the extreme by just treating the underlying
ring as field and by hopefully splitting the computation to smaller problems in case it fails.
To be more precise, consider a quotient $\R/n\R$ of a principal ideal domain $\R$ for some
non-trivial element $n \in \R$.
If $I \subseteq (\R/n\R)[x]$ is an ideal for which we want to find a strong Gröbner basis computation,
we pretend that $n$ is prime, that is, $\R/n\R$ is a field and apply a classical Gröbner basis algorithm
from the field case to $I$. If this does not encounter a non-invertible element, then we are done.
Otherwise we use a non-invertible element to split $n = a b$ with coprime elements $a, b \in R$.
After computing strong Gröbner bases of $I$ over $\R/a\R$ and $\R/b\R$, we pull them back along the canonical isomorphism
$\R/n\R \to \R/a\R \times \R/b\R$ to obtain a strong Gröbner basis of $I$.
In case we cannot split $n$, we fall back to a classical algorithm for computing strong Gröbner basis.
The most favorable case for the new algorithm are squarefree elements $n$, since then any non-invertible element allows us to split $n$.

The idea of working in $\R/n\R$ as if $n$ were a prime, is a common strategy in computer algebra.
Other examples include the computation of matrix normal forms over $\R/n\R$, see \cite{Fieker2014}.
To make this approach work in the setting of strong Gröbner bases, we investigate the behavior of strong Gröbner bases with respect to
quotients and the Chinese remainder theorem.
By properly normalizing the strong Gröbner basis, we prove that one can efficiently pull back strong Gröbner bases along a projection $\R \to \R/n\R$ (Theorem~\ref{thm:pir1}) as well as along a canonical isomorphism $\R/n\R \to \R/a\R \times \R/b\R$ (Theorem~\ref{thm:pir2}).

The algorithm has been implemented to compute strong Gröbner bases over residue class rings of the form $\Z/n\Z$, where $n \in \Z_{>0}$, see Section~\ref{sec-results}.
Running standard benchmarks for Gröbner basis computations for $n$ of different shape shows a consistent speed-up across all examples (except one). In case of squarefree $n$, the new algorithm improves upon the state of the art implementations by a factor of 10--100.

\section*{Acknowledgments}
This work was supported by DFG project SFB-TRR 195.
The authors thank Claus Fieker for helpful comments.

\section{Basic notions}
\label{sec:notation}
Let \R be a principal ideal ring, that is a unital commutative ring such that every ideal is principal. Note that
\R is not necessarily an integral domain.
If $a, b \in \R$ are two elements with $a\R \subseteq b\R$ we denote by $a/b$ by
abuse of notation any element $c \in \R$ with $c \cdot b = a$.
Recall that a \textit{least common multiple} of two elements $a, b$ is an element $l \in
\R$ such that $l\R = a\R \cap b\R$. By abuse of notation we denote by $\lcm(a, b)$
such an element. Similarly, we denote by $\gcd(a, b)$ an element of $\R$ with
$r\R + s\R = \gcd(a, b)\R$ and call it a \textit{greatest common divisor}.
For an element $n \in \R$ we denote by $\pi_n \colon \R \to \R/n\R$ the canonical projection.
For an ideal $I \subset \R$ we define the \emph{annihilator of $I$} by
\[\ann I = \left\{c \in \R \mid c \cdot a = 0\;\forall a \in I \right\}.\]
For an element $a\in\R$ we denote by $\ann a := \ann{\langle a \rangle}$ the \textit{annihilator of $a$}.


A \emph{polynomial}
in $n$ variables $x_1,\ldots,x_n$ over \R is a finite \R-linear
combination of \emph{terms} $\const_{v_1,\ldots,v_n} \prod_{i=1}^n
x_i^{v_i}$,
\[f=\sum_{v}\const_v x^v \defeq
\sum_{v\in\N^n}^{\text{finite}}\const_{v_1,\ldots,v_n} \prod_{i=1}^n
x_i^{v_i},\]
such that $v \in \N^n$ and $\const_v \in \R$.
The \emph{polynomial ring} $\R[x] \defeq \R[x_1,\ldots,x_n]$
in $n$ variables over $\R$ is the set of all polynomials over $\R$
together with the usual addition and multiplication. For $f=\sum_{v}\const_v
x^v \neq 0 \in\PR$ we define \emph{the degree of $f$} by $\deg(f) :=
\max\left\{v_1+\cdots +v_n \mid v \in \N^n, \const_v \neq 0\right\}$. For $f=0$ we set
$\deg(f):=
-1$.
We fix once and for all a monomial order $<$ on $\PR$, which, for the sake
of simplicity, is assumed to be global, that is, $x^\alpha \geq 1$ for all
$\alpha \in \N^n$.
Given a monomial order $<$ we can highlight the maximal terms of
elements in \PR with respect to $<$: For $f\in \PR\ \backslash\ \{0\}$,
$\hd f$ is the \emph{lead term}, $\hm f$ the \emph{lead monomial}, and $\lc f$
the \emph{lead coefficient} of $f$. For any set $F \subset \PR$ we define
the \emph{lead ideal} $L(F) = \langle \hd f \mid f \in F\rangle$; for an ideal
$I \subset \PR$,\  $L(I)$ is defined as the ideal of lead terms of all elements of
$I$.

%

The reduction process of two polynomials $f$ and $g$ in \PR depends now on the
uniqueness of the minimal remainder in the division algorithm in \R:

\begin{definition}
\label{def:reduction}
Let $f, g \in \PR$ and let $G= \{g_1,\ldots,g_r\} \subset \PR$ be a finite set of
polynomials.
\begin{enumerate}
\item We say that \emph{$g$ top-reduces $f$} if $\hm g \mid \hm f$ and $\lc g \mid \lc f$.
A top-reduction of $f$
by $g$ is then given by
    \[f - \frac{\lc f}{\lc g} \frac{\hm f}{\hm g} g.\]
\item Relaxing the reduction of the lead term to any term of $f$, we say
that \emph{$g$ reduces $f$}. In general, we speak of a reduction of a
polynomial $f$ with respect to a finite set $F\subset \PR$.
Let 
\item We say that $f$ has a \emph{weak standard representation} with respect to $G$ if
$f = \sum_{i=1}^r h_i g_i$ for some $h_i \in \PR$ such that $\hm f = \hm{h_j
  g_j}$ for some $j \in \{1,\ldots,r\}$.
\item We say that $f$ has a \emph{strong standard representation} with respect to $G$ if
$f = \sum_{i=1}^r h_i g_i$ for some $h_i \in \PR$ such that $\hm f = \hm{h_j
  g_j}$ for some $j \in \{1,\ldots,r\}$ and $\hm f > \hm{h_k g_k}$ for all $k
  \neq j$.
\end{enumerate}
\end{definition}

This kind of reduction is equivalent to definition CP3
from~\cite{kapur-cp3} and generalizes
Buchberger's attempt from~\cite{Buchberger1985}.
The result of such a reduction might not be unique. This uniqueness is exactly
the property \emph{\gbs} give us.

\begin{definition}
\label{def:strong-gb}
A finite set $G \subset \PR$ is called a \emph{\gb} for an ideal $I$ (with respect to
$<$) if $G \subset I$ and $L(G) = L(I)$. Furthermore, $G$ is called a
\emph{\sgb}
   if for any $f \in I\backslash\{0\}$ there exists an element $g\in G$
such that $\hd g \mid \hd f$. 
\end{definition}

\begin{remark}
Note that $G$ being a \sgb is equivalent to all elements $g \in G$ having
a strong standard representation with respect to $G$. See, for
example, Theorem~1 in~\cite{lichtblau2012} for a proof.
\end{remark}

Clearly, assuming that $\R$ is a field, any \gb is a \sgb. But in our setting
with \R being a principal ideal ring one has to check the coefficients, too, as
explained in Definition~\ref{def:reduction}.
The fact that for an arbitrary principal ideal ring the notions of Gröbner bases and strong Gröbner 
bases do not agree, can be observed already for monomial ideals in univariate polynomial rings:

\begin{example}
\label{ex:strong-gb}
Let $\R = \Z$ and $I= \langle x \rangle \in \R[x]$. Clearly, $G :=
\{2x,3x\}$ is a \gb for $I$: $L(I) = \langle x \rangle$ and $x = 3x-2x \in
L(G)$. But $G$ is not a \sgb for $I$ since $2x \nmid
x$ and $3x \nmid x$.
\end{example}

In order to compute \sgbs we need to consider two different types of
special polynomials:

\begin{definition}
\label{def:spoly}
Let $f,g \in \PR$, $t = \lcm\left(\hm f, \hm g\right)$, $t_f = \frac{t}{\hm f}$, and $t_g = \frac{t}{\hm
  g}$.
\begin{enumerate}
\item Let $a = \lcm\left(\lc f, \lc g\right)$, $a_f = \frac{a}{\lc g},$ and $a_g =
\frac{a}{\lc f}$. A \emph{\spt} of $f$ and $g$ is denoted by
\[\spoly f g = a_f t_f f - a_g t_g g.\]
\item Let $b = \gcd\left(\lc f, \lc g\right)$. Choose $b_f, b_g \in \R$ such that $b = b_f \lc f + b_g \lc
g$.
A \emph{\gpt} of $f$ and $g$ is denoted by
\[\gpoly f g = b_f t_f f + b_g t_g g.\]
\item Let $a \in \R$ be a generator of $\ann{\lc f}$. An \emph{annihilator polynomial} of $f$ is denoted by
\[\apoly f = a f = a \, \tail f.\]
\end{enumerate}
\end{definition}

\begin{remark} \
\label{rem:gpairs}
\begin{enumerate}
\item Note that $\spoly f g, \gpoly f g$ as well as $\apoly f g$ are not uniquely defined, since quotients, Bézout coefficients and generators are in general not unique.
\item If $\lc f$ is not a zero divisor in $\R$ then $\apoly f = 0$. It follows
that if $\R$ is a domain, there is no need to handle \apts since $0$ is the only
zero divisor.
\item In the field case we do not need to consider \gpts at all since we can
always normalize the polynomials, that is, ensure that $\lc f = 1$.
\end{enumerate}
\end{remark}

From Example~\ref{ex:strong-gb} it is clear that the usual Buchberger
algorithm as in the field case will not compute a \sgb as we would only
consider $\spoly{2x}{3x} = 3 \cdot 2x - 2 \cdot 3x = 0$. Luckily, we
can fix this via taking care of the corresponding \gpt:
\[\gpoly{2x}{3x} = (-1) \cdot 2x - (-1) \cdot 3x = x.\]
It follows that given an ideal $I \subset \PR$ a \sgb for $I$ can be achieved using
a generalized version of Buchberger's algorithm computing not only strong
standard representations of \spts
but also of \gpts and \apts. We refer,
for example,  to~\cite{lichtblau2012} for more details.

So, how do we get a strong standard representations of elements w.r.t. some set $G$?
The answer is given by the concept of a \emph{normal
form}:

\begin{definition}
Let $\mathcal G$ denote the set of all finite subsets $G \subset \PR$. We call
the map
$\text{\sffamily{NF}}: \PR \times \mathcal G \longrightarrow \PR, \, (f, G)
\longmapsto \nf f G$,
  a \emph{weak normal form} (w.r.t. a monomial ordering  $<$) if for all $f\in \PR$ and all $G \in
\mathcal G$ the following hold:
\begin{enumerate}
\item $\nf 0 G = 0$.
\item If $\nf f G \neq 0$ then $\hd{\nf f G} \notin L(G)$.
\item If $f \neq 0$ then there exists a unit $u\in \PR$ such that either
$uf = \nf f G$ or
$r = uf - \nf f G$ has a strong standard representation with respect to $G$.
\end{enumerate}
A weak normal form {\sffamily NF} is called a \emph{normal form} if we can always
choose $u=1$.
\end{definition}

Algorithm~\ref{alg:nfg} presents a normal form
algorithm for computations:

\begin{algorithm}[ht]
\caption{Normal form with respect to a global monomial order $<$
  (\nfn)} 
\label{alg:nfg}
\begin{algorithmic}[1]
\Require{Polynomial $f \in \PR$, finite subset $G\subset \PR$}
\Ensure{\nfn of $f$ w.r.t. $G$ and $<$}
\State{$h \gets f$}
\While{$\left(h \neq 0 \text{ and }G_h := \{g \in G \mid g \text{ top-reduces } h\}
  \neq\emptyset\right)$}
\State{Choose $g \in G_h$.}
\State{$h\gets $ Top-reduction of $h$ by $g$ (see
    Definition~\ref{def:reduction})}
\EndWhile
\State{$\text{\textbf{return }}h$}
\end{algorithmic}
\end{algorithm}

Now we state Buchberger's algorithm for computing \sgbs,
Algorithm~\ref{alg:bba}. For the theoretical background we refer to~\cite{gpSingularBook2007}
and~\cite{Becker1993}.

\begin{breakablealgorithm}
\caption{Buchberger's algorithm for computing \sgbs
  (\sbba)} 
\label{alg:bba}
\begin{algorithmic}[1]
\Require{Ideal $I=\langle f_1,\ldots,f_m\rangle \subset \PR$, normal form algorithm \nfn (depending on $<$)}
\Ensure{\gb $G$ for $I$ w.r.t. $<$}
\State{$G \gets \{f_1,\ldots,f_m\}$}
\State{$P \gets \left\{\apoly{f_i}  \mid 1 \leq i \leq n \right\}$}
\State{$P \gets \left\{\spoly{f_i}{f_j}, \gpoly{f_i}{f_j} \mid 1 \leq i < j \leq
  m\right\}$}\label{alg:bba:update1}
\While{$\left(P \neq \emptyset\right)$}
\State{Choose $h \in P$, $P \gets P \setminus \{h\}$}\label{alg:bba:choose}
\State {$h \gets \nf h G$}
\If{$\left(h \neq 0\right)$}
\State{$P \gets P \cup \left\{\apoly{h}\right\}$}
\State{$P \gets P \cup \left\{\spoly{g}{h}, \gpoly{g}{h} \mid g \in
  G\right\}$}\label{alg:bba:update2}
\State{$G \gets G \cup \{h\}$}\label{alg:bba:new-poly}
\EndIf
\EndWhile
\State{$\text{\textbf{return }}G$}
\end{algorithmic}
\end{breakablealgorithm}

\section{Strong \gbs over principal ideal rings}
\label{sec-pip}
In this section we give theoretical results for the computation of \sgbs over
principal ideal rings. These results will then be used in
Section~\ref{sec:algorithmic} for an improved computation of \sgbs over
quotients of principal ideal rings.
We begin by analyzing Algorithm~\ref{alg:bba} in case all occurring leading coefficients are invertible.

\begin{lemma}
\label{lem:invertible}
Let $I=\langle f_1,\ldots, f_m\rangle \subset \PR$ be an ideal such that for all
$i=1,\ldots, m$ we have that  $\lc{f_i}$ is invertible in $\R$. Moreover, assume that for each
newly added polynomial $h$ in Line~\ref{alg:bba:new-poly} in Algorithm~\ref{alg:bba} the polynomial $\lc h$ is
invertible in $\R$. Then Algorithm~\ref{alg:bba} does not need to consider \gpts and
\apts.
\end{lemma}

\begin{proof}
We show that all \gpts and all \apts are zero in the setting of the lemma:
\begin{enumerate}
\item For each element $g$ in the intermediate \gb $G$ it holds that $\lc g$ is
invertible in $\R$ and thus, not a zero divisor. It follows that $\apoly g = 0$
by definition.
\item For each $\gpoly f g$ for $f,g \in G$ it holds that $\lc f \mid \lc g$:
$\lc f$ is invertible in $\R$, so we get
\[\left(\lc g \cdot \left(\lc f\right)^{-1}\right) \lc f = \lc g.\]
Again, by definition, $\gpoly f g = 0$.\qedhere
\end{enumerate}
\end{proof}

\begin{remark}
From Lemma~\ref{lem:invertible} it follows that as long as
Algorithm~\ref{alg:bba} does not encounter a lead coefficient that is not
invertible in $\R$ we can use Buchberger's algorithm from the field case
without the need to consider \gpts and \apts for strongness properties. In
Section~\ref{sec:algorithmic} we discuss
how one can use this fact to improve the general computations of \sgbs over $\PR$.
\end{remark}

We next next show how to pull back a strong Gröbner bases along a canonical projection $\PR \to (\R/n\R)[x]$.

\begin{theorem}\label{thm:pir1}
  Let $n \in \R$, $n \neq 0$ and $I \subseteq \R[x]$ an ideal. Assume that $G_n
  \subseteq  \R[x]$ is a set of polynomials with the following properties:
  \begin{enumerate}
    \item $\pi_n(G_n)$ is a strong Gröbner basis of $\pi_n(I)$;
    \item for every $g \in G_n$ the leading coefficient $\lc g$ divides $n$ and
    $\lc g \not\in n\R$.
  \end{enumerate}
  Then $G_n \cup \{n\}$ is a strong Gröbner basis of $I + n\R[x]$.
\end{theorem}

\begin{proof}
  It is clear that $G_n \cup \{ n \} \subseteq I + n\R[x]$.
  Now let $f \in I$. If $\pi_n(\lc f) = 0$, then the leading term is a multiple of $n$.
  Thus we may assume that $\pi_n(f) \neq 0$. Since $\pi_n(G_n)$ is a strong Gröbner basis of $\pi_n(I)$, there exists $g \in G_n$ such that $\hd{\pi_n(g)}$ divides $\hd{\pi_n(f)}$.
  Hence we can find $h \in \R[x]$ with $\pi_n(h) \cdot \hd{\pi_n(g)} = \hd{\pi_n(f)}$.
  We can assume that $h$ is a term and $\hm h  \cdot \hm g  = \hm f $.
  By assumption we have $\pi_n(\hd g ) = \hd{\pi_n(g)}$ as well as $\pi_n(\hd f ) = \hd{\pi_n(f)}$.
  Hence we can write $h \cdot \hd g - \hd f = c \cdot x^{\deg(f)}$ for some $c
  \in n\R$.
  Since $\lc g$ divides $n$ and $\hm g$ divides $\hm f$ it follows that $\hd g$ divides $\hd f$.
\end{proof}

\begin{remark}
Assume that we know that an ideal $I \subseteq \R[x]$ contains a constant polynomial $n \in \R$, $n \neq 0$.
As $I = I + n\PR$, Theorem~\ref{thm:pir1} implies that we can compute a strong Gröbner basis
of $I$ be properly choosing the lifts of a strong Gröbner basis of the reduction
$\pi_n(I) \subseteq (\R/n\R)[x]$.
For $\R = \Z$, a similar idea can be found in Section~$4$
of~\cite{epp-sb-euclidean-domains-2018}. There it is described how to check if
an ideal $I \subseteq \Z[x]$ contains a constant polynomial $n \in \Z$, $n \neq
0$. In case it exists, the authors describe an ad hoc method which keeps the size
of the coefficients of the polynomials in Algorithm~\ref{alg:bba} bounded by
$n$.
\end{remark}

We now consider the following situation. Assume that $a, b, u, v \in \R$ are elements with $1 = ua + vb$ and $a b = 0$.
Note that this implies $(ua)^2 = ua$ and $(vb)^2 = vb$.

\begin{theorem}\label{thm:pir2}
  Assume that $I \subseteq \R[x]$ is an ideal. Furthermore let $\kidx, \lidx$ be finite index
  sets and $G_a = (g_{a, k})_{k \in \kidx},
  G_b = (g_{b, l})_{l \in \lidx}$ strong Gröbner bases of $I + a\R[x]$ and $I +
  b\R[x]$ respectively, satisfying the following conditions:
  \begin{enumerate}
    \item
      For $k \in \kidx$, if $g_{a, k}$ is non-constant, then $\lc{g_{a, k}}$ divides
      $a$ and $\lc{g_{a, k}} \not\in a\R$
    \item
      For $l \in \lidx$, if $g_{b, l}$ is non-constant, then $\lc{g_{b, l}}$ divides
      $b$ and $\lc{g_{b, l}} \not\in b\R$.
  \end{enumerate}
  For $k \in \kidx$, $l \in \lidx$ define
  \begin{align*} f_{k, l} = &ua \frac{\lcm\left(\hm{g_{a, k}},\hm{g_{b,
          l}}\right)}{\hm{g_{b, l}}} \cdot \lc{g_{a, k}} g_{b, l} \\ 
  + &vb \frac{\lcm\left(\hm{g_{a, k}},\hm{g_{b, l}}\right)}{\hm{g_{a, k}}} \cdot \lc{g_{b, l}} g_{a, k}. \end{align*}
  Then $G = \{ f_{k, l} \mid k \in \kidx, l \in \lidx\}$
  is a strong Gröbner basis of $I$.
\end{theorem}

\begin{proof}
  Note that from $1 = ua + vb$ it follows at once that $I = eaI + vbI$.
  Hence $I = eaI + vbI = ea(I + b\R[x]) + vb(I + a\R[x]) = ua \langle G_b \rangle + vb \langle G_a \rangle$.
  As $\lc{g_{a, k}}$ and $\lc{g_{b, l}}$ are coprime we have $\lcm(\lc{g_{a,
          k}}, \lc{g_{b, l}}) = \lc{g_{a, k}} \lc{g_{b, l}}$.
  Moreover, since
  \[ (\lc{g_{a, k}}\lc{g_{b, l}})\R = (\lc{g_{a,k}})\R \cap (\lc{g_{b, l}})\R
    \supsetneq a\R \cap b\R = \{ 0\} \]
  we have
  \[ ua \cdot \lc{g_{a, k}} \lc{g_{b, l}} + vb \cdot \lc{g_{b, k}} \lc{g_{a, l}}
  = \lc{g_{a,k}} \lc{g_{b, l}} \neq 0. \]
  In particular
  \[ \hd{f_{k, l}} = \lcm\left(\hm{g_{a,k}}, \hm{g_{b, l}}\right) \cdot \lc{g_{a,
      k}}\lc{g_{b, l}}. \]

  Consider now an element $h \in I$. Since $h \in I + a\R[x]$ and $h \in I +
  b\R[x]$, there exist $g_{a, k} \in G_a$, $g_{b, l} \in G_b$ such that
  $\hd{g_{a, k}} \mid \hd h$ and $\hd{g_{b, l}} \mid \hd h$.
  Thus $\hm h$ is divisible by $\lcm\left(\hm{g_{a, k}},\hm{g_{b, l}}\right)$ and $\lc h$ is
  divisible by $\lc{g_{a, k}} \lc{g_{b, l}}$, that is,
  $\hd{h}$ is divisible by $\hd{f_{k, l}}$.
\end{proof}

\section{Algorithmic approach for computing in quotients of principal ideal rings}
\label{sec:algorithmic}
We now assume that $\R$ is a principal ideal domain. Additionally we now also fix an element $n \in \R$, $n \neq 0$.
Using the theoretical results from Section~\ref{sec-pip} we are now able to
describe improvements to the \gb computation over the base ring $\R/\n\R$.

\begin{corollary}\label{cor-recombine}
  Let $I \subseteq (\R/n\R)[x]$ be an ideal and $n = a \cdot b$ a factorization
  of $n$ into coprime elements $a, b \in \R$. Let $\kidx, \lidx$ be finite index
  sets.
  Assume that $G_a = (g_{a, k})_{k \in \kidx} \subseteq (\R/n\R)[x]$ is a set of polynomials, such that
  $\pi_a(G_a) \subseteq (\R/a\R)[x]$ is a strong Gröbner basis of $\pi_a(I) \subseteq (\R/a\R)[x]$, $a \in G_a$ and for every $g \in G_a \setminus \{a\}$,
  the leading coefficient $\lc g$ divides $a$ and is not divisible by $a$.
  Assume that $G_b = (g_{b, l})_{l \in \lidx} \subseteq \R/n\R[x]$ has similar properties with respect to $b$.
  For $i \in I$, $j \in J$ define
  \begin{align*} f_{k, l} = &ua \frac{\lcm(\hm{g_{a, k}},\hm{g_{b, l}})}{\hm{g_{b, l}}} \cdot \lc{g_{a, k}} g_{b, l} \\ 
  + &vb \frac{\lcm(\hm{g_{a, k}},\hm{g_{b, l}})}{\hm{g_{a, k}}} \cdot \lc{g_{b, l}} g_{a, k}. \end{align*}
  Then $G = \{ f_{k, l} \mid k \in \kidx, l \in \lidx\}$
  is a strong Gröbner basis of $I$.
\end{corollary}

\begin{proof}
  Follows at once from Theorems~\ref{thm:pir1} and~\ref{thm:pir2} since for
  $\bar a \in \bar \R = \R/n\R$ we have $\bar \R/\bar a \bar \R \cong \R/a \R$.
\end{proof}

To use this, we need, given a divisor $a \in \R$ of $n$, a way to lift polynomials from $\R/a\R$ to $\R/n\R$ such that the leading coefficients divide $n$.

\begin{lemma}
  There exists an algorithm, that given $c \in \R$ determines $u \in \R$ such that $\gcd(u, n) \in \R^\times$ and $uc = \gcd(c, n) \bmod n$.
\end{lemma}

\begin{proof}
  This can be found in \cite[Section~2]{Storjohann1998}.
\end{proof}

In case we have an algorithm for factoring elements of $\R$ into irreducible elements, this allows us to reduce the
strong Gröbner basis computation to computations over smaller quotient rings.
This approach is summarized in Algorithm~\ref{alg-stronggbRnRnaive}.

\begin{algorithm}
\caption{Naive strong Gröbner basis over $\R/n\R$} 
\label{alg-stronggbRnRnaive}
\begin{algorithmic}[1]
  \Require{Ideal $I = \langle f_1 \bmod n,\dotsc,f_m \bmod n \rangle \subset (\R/n\R)[x]$.}
\Ensure{Strong Gröbner basis $G$ for $I$}
  \State{Factor $n = p_1^{e_1} \dotsm p_r^{e_r}$, with pairwise coprime irreducible elements $p_i \in \R$.}
  \State{For $1 \leq i \leq r$ compute strong Gröbner bases $G_i$ of $I_i = \langle f_1 \bmod p_i^{e_i}, \dotsc, f_m \bmod  p_i^{e_i} \rangle \subseteq (\R/p_i^{e_i})\R[x]$
         using Algorithm~\ref{alg:bba}.}
  \State{Apply Corollary~\ref{cor-recombine} recursively to obtain a strong Gröbner basis $G$ of $I$.}
\end{algorithmic}
\end{algorithm}

Depending on the ring $\R$, the particular $n$ and the factorization algorithm, Step~(1) is infeasible or not.
For example if $\R = \Z$, the fastest factorization algorithms are subexponential in $n$,
rendering this approach futile for non-trivial example with large $n$.
On the other hand, if $\R = \Q[t]$ or $\R = \F_p[t]$ for some prime $p$,
then factoring in $\R$ can be done in (randomized) polynomial time in the size of $n$ and thus is
a good idea, at least from a theoretical point of view.

We now consider the case, where we cannot or do not want to factor the modulus $n$.
The basic idea is to run the algorithm from the field case, pretending that
$\R/n\R$ is a field, and to stop whenever we find a non-invertible leading
coefficient. If the algorithm discovers a non-invertible element, we try to split the
modulus and the computation of the strong Gröbner basis.
The splitting is based on the following consequence of so-called factor refinement.

\begin{proposition}\label{prop-split}
  There exists an algorithm that given $a \in \R$ with $\gcd(a, n) \not\in \R^\times \cup n\R$, that is, $0 \neq \bar a \notin (\R/n\R)^\times$, either
  \begin{enumerate}
    \item
      finds $m \in \R$, $k \in \Z_{> 1}$ with $n = m^k$ and $m = \gcd(a, m)$, or
    \item
      finds coprime elements $p, q \in \R \setminus \R^\times$ with $n = p \cdot q$.
  \end{enumerate}
\end{proposition}

\begin{proof}
  Using an algorithm for factor refinement, for example the algorithm of Bach--Driscoll--Shallit (see \cite{Bach1993}), we can find a set $S \subseteq \R \setminus \R^\times$ of coprime elements such that
  $a$ and $n$ factor uniquely into elements of $S$ and for all $s \in S$ we have $\gcd(a, s)$ or $\gcd(b, s)$ not in $\R^\times$.
  Now pick $m \in S$ with $m \mid n$.
  We can write $n = m^k \cdot u$ with $k \in \Z_{>0}$ and $u \in \R$ coprime to $m$.
  Depending on whether $u$ is a unit or not, we are in case (1) or (2).
\end{proof}

Incorporating this into the strong Gröbner basis computation we obtain Algorithm~\ref{alg-stronggbRnR}.

\begin{algorithm}
\caption{Strong Gröbner basis over $\R/n\R$} 
\label{alg-stronggbRnR}
\begin{algorithmic}[1]
  \Require{Ideal $I = \langle f_1 \bmod n,\dotsc,f_m \bmod n \rangle \subset \R/n\R[x]$, monomial order $<$.}
\Ensure{Strong Gröbner basis $G$ for $I$ w.r.t. $<$}
  \State{Apply Algorithm~\ref{alg:bba} to $I$ and stop whenever there is a non-invertible lead coefficient $\overline a \in \R/n\R$.}\label{step-1}
  \If{Step~\ref{step-1} returned $G \subseteq (\R/n\R)[x]$}
  \State{Return $G$.}
  \Else
  \State{Step~\ref{step-1} returned $a \in \R$ with $\gcd(a, n) \not\in \R^\times \cup n\R$.}
  \State{Apply Proposition~\ref{prop-split} to the pair $(a, n)$}\label{step-2}
  \If{Step~\ref{step-2} returned $n = p \cdot q$ with coprime $p, q$}
    \State{Set $I_p = \langle f_1 \bmod p, \dotsc, f_m \bmod p\rangle \subseteq (\R/p\R)[x]$.}
    \State{Set $I_q = \langle f_1 \bmod q, \dotsc, f_m \bmod q\rangle \subseteq (\R/q\R)[x]$.}
    \State{Apply Algorithm~\ref{alg-stronggbRnR} to obtain strong Gröbner bases $G_p$ and $G_q$ of $I_p$ and $I_q$ respectively.}
    \State{Use Corollary~\ref{cor-recombine} to obtain a strong Gröbner basis $G$ of $I$ and return $G$.}
  \Else
    \State{Apply Algorithm~\ref{alg:bba} to $I$ and return the resulting strong Gröbner basis of $I$.}
  \EndIf
  \EndIf
\end{algorithmic}
\end{algorithm}

\begin{theorem}
  Algorithm~\ref{alg-stronggbRnR} terminates and is correct.
\end{theorem}

\begin{proof}
  Termination follows since in the recursion, the number of irreducible factors of $n$ is strictly decreasing.
  Correctness follows from Corollary~\ref{cor-recombine}.
\end{proof}

\begin{remark}
  The usefulness of the splitting depends very much on the factorization of $n$.
  For example, if $n = p^e$ is the power of an irreducible element $p$, then Algorithm~\ref{alg-stronggbRnR}
  is the same as Algorithm~\ref{alg:bba}.
  The most favorable input for Algorithm~\ref{alg-stronggbRnR} are rings $\R/n\R$ with $n$ squarefree, that is, $n = p_1\dotsc p_r$
  is the product of pairwise coprime irreducible elements (including the case, where $n = p_1$ is itself irreducible).
  In this case, every non-invertible element allows us to split the modulus into coprime elements.
  Thus all Gröbner basis computations can be done as in the field case.
\end{remark}

\section{Experimental results}
\label{sec-results}
In the following we present experimental results comparing our new approach to
the current implementations in the computer algebra systems \singular
(\cite{singular412}) and \magma (\cite{bcpMagma}). All computations were done on an Intel\textsuperscript{\textregistered} Xeon\textsuperscript{\textregistered} CPU E5-2643 v3 @ 3.40GHz
with $384$GB RAM. Computations that took more than 24 hours were terminated by hand.

Our new algorithm is implemented in the \julia package \gbjl (\cite{gbjl}) which is part of the
\oscar project of the SFB TRR-195.
The package \gbjl is based on the C library \gbl
(\cite{gbl}) which implements Faug\`ere's F4 algorithm (\cite{ff41999}) for computing \gbs over
finite fields. The implementation of Algorithm~\ref{alg-stronggbRnR} uses \gbl and \singular as follows:
All the computations over a ring $\R/n\R$ for which we want to execute the algorithm from the field
case are delegated to \gbl. In case we find $n \in R$ such that $n$ is not prime but we cannot find a factorization
of $n$ into coprime elements (Proposition~\ref{prop-split}~(1)), we delegate the corresponding strong Gröbner basis computation to \singular.
All the lifting and recombination steps are done in \julia (see~\cite{Fieker2017}).
Note that in practice we always compute minimal strong Gröbner bases and
make sure that minimality is preserved during the recombination using Corollary~\ref{cor-recombine}.
This makes sure that for all intermediate Gröbner bases that we compute the size is bounded by
the size of a minimal strong Gröbner basis of the input.


We use a set of different benchmark systems focusing on pair handling, the
reduction process, finding of reducers, respectively. We have computed \sgbs for
these systems over $\Zn$ using three different settings for $\n$:
\begin{enumerate}
\item For $\n = 2\cdot 3 \cdot 5 \cdot 7 \cdot 11 \cdot 13 \cdot 17 \cdot 19 \cdot
23$ (Table~\ref{table:first-example}) we get a factorization down to the finite
prime fields. Thus in all theses examples our new implementation can use the
F4 algorithm implemented in \gbl as base case.
With ``*'' we highlight examples for which there was no non-invertible element
discovered, that is, the computation from the field case runs through without any
splitting of $\n$ to be considered.
\item For $\n = 32771 \cdot 32779$ (Table~\ref{table:second-example}) we can
see that our approach of applying Lemma~\ref{lem:invertible} is very promising:
In none of the examples tested we found non-invertible elements, thus we compute
the basis as if we are working over a finite field, receiving a correct \sgb over $\Zn$.
\item For $\n =3^3 \cdot 5^3 \cdot 7^3 \cdot 11^3$
(Table~\ref{table:third-example}) we, in general, have to use
\singular's \sgb algorithm for computing in $\Z_{p^k}$. Still, we can see that
our approach is most often by a factor of at least $3$ faster than directly
applying \singular's implementation over $\Zn$. The only exception is \textsc{Jason-210}, for which \singular
is faster: The basis is huge ($>2,009$ generators), thus our new implementation
needs roughly $75$ of the overall $171$ seconds to apply the recombination and
lifting due to Corollary~\ref{cor-recombine}.
Again, we highlight with ``*'' examples for which there was no non-invertible element
discovered.

\end{enumerate}
\begin{table}[h]
	\centering
  \def\arraystretch{1.2}
    \begin{tabular}{c||r|r|r}
    \toprule
    \multicolumn{1}{c||}{\textbf{Examples}} &
    \multicolumn{1}{c|}{\textbf{New Algorithm}} &
    \multicolumn{1}{c|}{\singular} &
    \multicolumn{1}{c}{\magma}\\
    \midrule
    Cyclic-6 & $0.233$ & $4.453$ & $1.170$ \\
    Cyclic-7 & $3.871$ & $1,689.805$ & $303.579$ \\
    Cyclic-8 & $26.211$ & $>24$h & $8,304.070$  \\[0.2em]
    Katsura-8 & $2.152$ & $89.154$ & $12.980$ \\
    Katsura-9 & $6.831$ & $1,592.969$ & $133.452$ \\
    Katsura-10 & $29.814$ & $21,316.400$ & $2,546.010$\\[0.2em]
    Eco-10 & $2.468$ & $67.440$ & $110.530$ \\
    Eco-11 & $12.329$ & $738.252$ & $2,799.549$ \\[0.2em]
    F-744 & $1.131$ & $61.754$ & $24.810$ \\
    F-855 & $4.879$ & $1,685.831$ & $>24$h \\[0.2em]
    Noon-7 & $2.309$ & $47.787$ & $388.050$ \\
    Noon-8 & $25.398$ & $2,809.623$ & $>24$h \\[0.2em]
    Reimer-5 & $0.749$ & $6.899$ & $4.030$ \\
    Reimer-6 & $5.748$ & $869.740$ & $1,282.950$ \\[0.2em]
    Lichtblau & $0.124$ & $2.904$ & $1.630$ \\[0.2em]
    Mayr-42 * & $32.672$ & $332.307$ & $199.360$ \\[0.2em]
    Yang-1 * & $54.851$ & $210.097$ & $115.130$ \\[0.2em]
    Jason-210 & $52.167$ & $178.348$ & $>24$h\\
    \bottomrule
    \end{tabular}
	\caption{Benchmark timings given in seconds, computations in $\Zn$ with $\n=2
  \cdot 3 \cdot 5 \cdot 7 \cdot 11 \cdot 13 \cdot 17 \cdot 19 \cdot 23$}
	\label{table:first-example}
\end{table}
\begin{table}[h]
	\centering
  \def\arraystretch{1.2}
    \begin{tabular}{c||r|r|r}
    \toprule
    \multicolumn{1}{c||}{\textbf{Examples}} &
    \multicolumn{1}{c|}{\textbf{New Algorithm}} &
    \multicolumn{1}{c|}{\singular} &
    \multicolumn{1}{c}{\magma}\\
    \midrule
    Cyclic-6 & $0.004$ & $0.237$     & $0.070$ \\
    Cyclic-7 & $0.179$ & $53.924$    & $3.130$ \\
    Cyclic-8 & $2.481$ & $5,970.100$ & $257.180$ \\[0.2em]
    Katsura-8 & $0.018$ & $1.252$    & $0.330$ \\
    Katsura-9 & $0.098$ & $11.066$   & $2.390$ \\
    Katsura-10 & $0.613$ & $113.251$ & $18.960$\\[0.2em]
    Eco-10 & $0.153$ & $15.824$      & $2.970$ \\
    Eco-11 & $0.909$ & $201.770$     & $22.360$\\[0.2em]
    F-744 & $0.016$ & $2.182$        & $0.250$ \\
    F-855 & $0.079$ & $11.939$       & $1.710$ \\[0.2em]
    Noon-7 & $0.171$ & $3.993$       & $1.140$ \\
    Noon-8 & $1.234$ & $78.320$      & $8.920$ \\[0.2em]
    Reimer-5 & $0.006$& $0.136$      & $0.130$ \\
    Reimer-6 & $0.066$ & $3.036$     & $2.320$ \\[0.2em]
    Lichtblau & $0.002$ & $0.041$    & $0.010$ \\[0.2em]
    Mayr-42 & $32.168$ & $309.142$   & $333.880$\\[0.2em]
    Yang-1 & $53.562$ & $193.240$    & $196.930$\\[0.2em]
    Jason-210 & $2.781$ & $12.576$   & $62.040$\\
    \bottomrule
    \end{tabular}
	\caption{Benchmark timings given in seconds, computations in $\Zn$ with
      $\n=32771 \cdot 32779$}
	\label{table:second-example}
\end{table}
\begin{table}[h]
	\centering
  \def\arraystretch{1.2}
    \begin{tabular}{c||r|r|r}
    \toprule
    \multicolumn{1}{c||}{\textbf{Examples}} &
    \multicolumn{1}{c|}{\textbf{New Algorithm}} &
    \multicolumn{1}{c|}{\singular} &
    \multicolumn{1}{c}{\magma}\\
    \midrule
    Cyclic-6 & $2.754$ & $5.374$ & $0.720$ \\
    Cyclic-7 & $679.141$ & $1,638.422$ & $ 1,209.289$ \\
    Cyclic-8 & $72,386.042$ & $>24$h & $9,471.340$ \\[0.2em]
    Katsura-8 & $27.054$ & $59.659$ & $40.050$ \\
    Katsura-9 & $433.774$ & $883.497$ & $599.736$ \\
    Katsura-10 & $1,239.260$ & $11,002.259$& $8.528.819$ \\[0.2em]
    Eco-10 & $28.426$ & $66.515$ & $574.949$ \\
    Eco-11 & $778.855$ & $1,058.806$ & $>24$h\\[0.2em]
    F-744 & $169.179$ & $343.612$ & $28.996.159$\\
    F-855 & $3,434.641$ & $21,941.538$ & $>24$h\\[0.2em]
    Noon-7 & $1,210.192$ & $13,291.706$ & $>24$h\\
    Noon-8 & $78,048.640$ & $>24$h & $>24$h\\[0.2em]
    Reimer-5 & $2.717$ & $8.061$ & $27.660$ \\
    Reimer-6 & $92.295$ & $927.624$ & $14,069.210$ \\[0.2em]
    Lichtblau & $0.279$ & $5.079$ & $4.060$ \\[0.2em]
    Mayr-42 * & $32.601$ & $362.176$ & $327.420$ \\[0.2em]
    Yang-1 * & $60.499$ & $220.546$ & $204.740$ \\[0.2em]
    Jason-210 & $171.565$ & $122.727$ & $>24$h\\
    \bottomrule
    \end{tabular}
	\caption{Benchmark timings given in seconds, computations in $\Zn$ with
      $\n=3^3 \cdot 5^3 \cdot 7^3 \cdot 11^3$}
	\label{table:third-example}
\end{table}

\section{Conclusion}
\label{sec:conclusion}
We have presented a new approach for computing \sgbs over principal ideal rings
$\R$
which exploits the factorization of composite moduli $\n$ to recursively compute
\sgbs in smaller rings and lifting the results back to $\R$. In many situations
the base cases of this recursive step boil down to computations over finite
fields which are much faster than those over principal ideal rings.

One further optimization of our new approach might be the following:
Once we have several factors of $\n$ found, we can run the
different, independent \gb computations in parallel. This is one of our next
steps. Another one is to implement an optimized version of Faug\`ere's F4
algorithm for $\R/p^k\R$ in \gbl.

\bibliographystyle{elsarticle-harv}
\bibliography{bib.bib}

\end{document}